\documentclass[a4paper,12pt]{article}

\usepackage{amsmath}
\usepackage{amssymb}
\usepackage{amsthm}
\usepackage{graphicx}

\newtheorem{theorem}{Theorem}[section]
\newtheorem{corollary}{Corollary}[section]

\theoremstyle{definition}

\newtheorem{remark}{Remark}[section]

\newcommand*{\Z}{\mathbb{Z}}
\newcommand*{\PP}{\mathbb{P}}

\def\id{\operatorname{id}}

\usepackage[usenames]{color}
\usepackage[colorlinks=true,linkcolor=blue,filecolor=red,
citecolor=webgreen]{hyperref}
\definecolor{webgreen}{rgb}{0,.5,0}

\numberwithin{equation}{section}

\begin{document}

\title{Inertia, positive definiteness and $\ell_p$ norm of GCD and LCM matrices and\\ their unitary analogs}

\author{Pentti Haukkanen\\
Faculty of Natural Sciences\\
FI-33014 University of Tampere, Finland\\
\\
L\'aszl\'o T\'oth\\
Department of Mathematics \\
University of P\'ecs,
Ifj\'us\'ag \'utja 6 \\
H-7624 P\'ecs, Hungary \\}
\maketitle

\begin{abstract}
Let $S=\{x_1,x_2,\dots,x_n\}$ be a set of distinct positive integers, and let $f$ be an arithmetical function. The GCD matrix $(S)_f$ on $S$ associated with $f$ is defined as
the $n\times n$ matrix having $f$ evaluated at the greatest common divisor of $x_i$ and $x_j$ as its $ij$ entry. The LCM matrix $[S]_f$ is defined similarly. We consider inertia, positive definiteness and $\ell_p$ norm of GCD and LCM matrices and their unitary analogs. Proofs are based on matrix factorizations and convolutions of arithmetical functions. 


\end{abstract}

\newpage

\tableofcontents

\section{Introduction}

Let $S=\{x_1,x_2,\dots,x_n\}$ be a set of distinct positive integers, and let $f$ be an arithmetical function. The GCD matrix $(S)_f$ on $S$ associated with $f$ is defined as
the $n\times n$ matrix having $f$ evaluated at the greatest common divisor of $x_i$ and $x_j$ as its $ij$ entry. The LCM matrix $[S]_f$ is defined similarly.
For $f(x)=x$ we obtain the usual GCD and LCM matrices $(S)$ and $[S]$ on $S$, and
for $f(x)=x^m$ we obtain the power GCD and LCM matrices on $S$.

H.J.S. Smith  calculated $\det (S)_f$ on factor-closed sets \cite[(5.)]{Sm} and
$\det [S]_f$ in a more special case \cite[(3.)]{Sm}.
There is a large number of generalizations and analogues of
these determinants in the literature.
For general accounts, see  \cite{HWS,SC}.
We assume that the reader is familiar with the
modern terminology of GCD and LCM matrices.

Various properties of GCD and LCM matrices and their analogues and generalizations are presented in the literature.
Since Smith, determinant, inverse and factorizations have been studied very extensively.
Lately, more attention is paid to eigenvalues, positive definiteness and norms.
Computational aspects are also brought forth to the agenda of the study of these type matrices
\cite{KDoc}.

In this paper we present some further results on inertia, positive definiteness and $\ell_p$ norm of GCD and LCM matrices and their unitary analogs.
The background material on unitary analogs is presented in Section \ref{se: pre}.
 Inertia means the numbers of positive, negative and zero eigenvalues and is explained in more detail in Section \ref{se: inertia}.
The study of inertia and positive definiteness utilizes the factorizations presented in Section  \ref{se: facto} and Sylvester's law of inertia.
The study of $\ell_p$ norm bases on convolutional methods, see Section~\ref{se: norms}.

\section{Preliminaries}\label{se: pre}

\subsection{The Dirichlet convolution}

The Dirichlet convolution of arithmetical functions $f$ and $g$ is defined as
$$
(f \star g)(n)= \sum_{d|n}f(d)g(n/d).
$$
The function $\delta$ (defined as $\delta(1)=1$ and $\delta(n)=0$ otherwise) serves as the identity under the Dirichlet convolution.
The M\"obius function $\mu$ is the inverse of the constant function $1$ under
the Dirichlet convolution.

Let $S=\{x_1,x_2,\dots,x_n\}$ be a set of distinct positive integers, and define
\begin{equation}
A_f(x_i)=\sum_{\substack{d\mid x_i\\d\nmid x_t\\t<i}} \phi_f(d)
\end{equation}
for all  $i = 1,2,\dots,n$, where
\begin{equation}
\phi_f(x) = \sum_{d\mid x}f(d)\mu(x/d)=(f\star\mu)(x).
\end{equation}
If $S$ is factor-closed, then $A_f(x_i)=\phi_f(x_i)$.
If $f(x)=x$ for all $x$, then $\phi_f=\phi$, Euler's totient function.
If $f(x)=x^m$ for all $x$, then $\phi_f=J_m$,
Jordan's totient function $J_m$.
General accounts on the Dirichlet convolution can be found in \cite{Mc,Si}.

\subsection{Unitary divisors and convolution}

A divisor $d\in\Z_+$ of $n\in\mathbb{Z_+}$ is said to be a unitary
divisor of $n$ and is denoted by $d\parallel n$ if $(d, n/d)=1$.
For example, the unitary divisors of $72$
$(=2^33^2)$ are $1, 8, 9, 72$.
If $d\parallel n$, we also say that $n$ is a unitary multiple of $d$.
The greatest common unitary divisor (GCUD) of $m$ and $n$ exists for all $m, n\in\Z_+$
but, unfortunately, the least common unitary multiple (LCUM) of $m$ and $n$ does not always exist.
For example, the LCUM of $2$ and $4$ does not exist.
The GCUD of $m$ and $n$ is denoted by $(m, n)^{\ast\ast}$ and the LCUM is denoted by  $[m, n]^{\ast\ast}$ when it exists.

Hansen and Swanson  \cite{HS} overcame the difficulty of the nonexistence of the LCUM by defining
\begin{equation}\label{eq: uulcm}
[m, n]^{\ast\ast}=\frac{mn}{(m, n)^{\ast\ast}}.
\end{equation}
It is easy to see that $mn/(m, n)^{\ast\ast}$ exists for all $m, n\in\Z_+$ and
is equal to the usual LCUM of $m$ and $n$ when the usual LCUM exists.
Therefore $[m, n]^{\ast\ast}$ in \eqref{eq: uulcm} is well-defined.
We say that $[m, n]^{\ast\ast}$
in \eqref{eq: uulcm} is the pseudo-LCUM of $m$ and $n$.
If the LCUM exists, then it is equal to the pseudo-LCUM.
There exist  also extensions of the LCUM other than the pseudo-LCUM
in the literature \cite{HINS}.

The unitary convolution of arithmetical functions $f$ and $g$ is defined as
$$
(f \oplus g)(n)= \sum_{d|n}f(d)g(n/d).
$$
The function $\delta$ also serves as the identity under the unitary convolution.
The unitary analog of the M\"obius function is the inverse of the constant function $1$ under
the unitary convolution and it is denoted by $\mu^\ast$.
The function $\mu^\ast$ is the multiplicative function such that
$\mu^\ast(p^k)=-1$ for all prime powers $p^k$ $(k\ge 1)$.

Let $S=\{x_1,x_2,\dots,x_n\}$ be a set of distinct positive integers, and define
\begin{equation}\label{eq: u1}
B_f^{\ast}(x_i)=\sum_{\substack{d\parallel x_i\\d\nparallel
x_t\\t<i}} \phi_f^{\ast}(d)
\end{equation}
for all  $i = 1,2,\dots,n$, where
\begin{equation}\label{eq: Jf}
\phi_f^{\ast}(x)
= \sum_{d\parallel x}f(d)\mu^*(x/d)
=(f\oplus\mu^*)(x).
\end{equation}
If $S$ is unitary divisor (UD) -closed, then $B_f^{\ast}(x_i)=\phi_f^{\ast}(x_i)$.
If $f(x)=x$ for all $x$, then $\phi_f^{\ast}=\phi^{\ast}$,
the unitary analog of Euler's totient function $\phi$, see \cite{C,Nag}.
If $f(x)=x^m$ for all $x$, then $\phi_f^{\ast}=J_m^{\ast}$,
the unitary analog of Jordan's totient function $J_m$, see \cite{Nag,ST}.

\subsection{Quasimultiplicative functions}

An arithmetical function $f$ is said to be quasimultiplicative
if $f(1)\neq 0$ and
\begin{equation}\label{eq: gam}
f(1)f(mn)=f(m)f(n)
\end{equation}
for all $m,n \in \Z_+$ with $(m,n)=1.$
A quasimultiplicative function $f$ is multiplicative if and only if $f(1)=1.$
An arithmetical function $f$ with $f(1)\neq 0$ is quasimultiplicative if and only if $f/f(1)$ is multiplicative.
Completely multiplicative functions are multiplicative functions satisfying
$f(mn)=f(m)f(n)$ for all $m,n \in \Z_+$.
General accounts on multiplicative functions are presented in \cite{Mc,Si}.

\subsection{GCD type matrices and their unitary analogs}

Let $S=\{x_1,x_2,\dots,x_n\}$ be a set of distinct positive integers, and let $f$ be an arithmetical function.
The GCD matrix $(S)_f$ and the LCM matrix $[S]_f$ are defined in Section~1.
Their unitary analogs go as follows.
The $n\times n$ matrix having
$f((x_i,x_j)^{\ast\ast})$ as its $ij$ entry is denoted as $(S^{\ast\ast})_f$,
and similarly the $n\times n$ matrix having
$f([x_i,x_j]^{\ast\ast})$ as its $ij$ entry is denoted as  $[S^{\ast\ast}]_f$.
We say that these matrices are  the GCUD and the pseudo-LCUM matrices on $S$
with respect to $f$.
For $f(x)=x$ we obtain the usual GCUD and pseudo-LCUM matrices
$(S^{\ast\ast})$ and $[S^{\ast\ast}]$ on $S$, and
for $f(x)=x^m$ we obtain the power GCUD and power-pseudo-LCUM matrices on $S$.
For general accounts on GCD type matrices see \cite{HWS,SC}.

\section{Factorizations} \label{se: facto}

We first review some factorizations presented in \cite{HaS}.

Let $S=\{x_1, x_2, \ldots, x_n\}$
be a GCD-closed set of distinct positive integers,
and let $f$ be any arithmetical function.
Then
\begin{equation}\label{eq: delta}
(S)_f=E\Delta E^T,
\end{equation}
where $E$ and $\Delta={\rm diag}(\delta_1, \delta_2,\ldots, \delta_n)$
are the $n\times n$ matrices defined by
$$
e_{ij}=
\begin{cases}
1\ {\rm if}\  \ x_j\,\vert\, x_i,\\
0\   {\rm otherwise},
\end{cases}
$$
and
$$
\delta_i=A_f(x_i)=\sum_{d\vert x_i\atop {d\,\not\>\vert\,x_t\atop x_t<x_i}}
(f\star\mu)(d).
$$
Further, if $f$ is a quasi-multiplicative function such that $f(x)\ne 0$ for all $x$,
then
\begin{equation}\label{eq: deltap}
[S]_f=\Lambda E\Delta^\prime E^T\Lambda,
\end{equation}
where $\Lambda$ and $\Delta'$ are the $n\times n$ diagonal matrices,  whose diagonal elements are
$\lambda_i=f(x_i)$ and
$$
\delta_i^\prime=A_{1/f}(x_i)=\sum_{d\vert x_i\atop {d\,\not\>\vert\,x_t\atop x_t<x_i}}
(\frac{1}{f}\star\mu)(d).
$$

Let $S=\{x_1,x_2,\ldots, x_n\}$ be a GCUD-closed set
of distinct positive integers, and
let $f$ be any arithmetical function.
Then
\begin{equation}\label{eq: gamma}
(S^{\ast\ast})_f=U\,\Gamma\,U^T,
\end{equation}
where $U$ and $\Gamma={\rm diag}(\gamma_1, \gamma_2,\ldots, \gamma_n)$
are the $n\times n$ matrices defined by
$$
u_{ij}=
\begin{cases}
1\ {\rm if}\  \ x_j\,\Vert\, x_i,\\
0\   {\rm otherwise},
\end{cases}
$$
and
$$
\gamma_i=B_f^{\ast}(x_i)=\sum_{d\Vert x_i\atop {d\,\not\>\Vert\,x_t\atop x_t<x_i}}
(f\oplus\mu^\ast)(d).
$$
Further, if  $f$ is a completely multiplicative function
such that $f(x)\ne 0$ for all $x$, then
\begin{equation}\label{eq: gammap}
[S^{\ast\ast}]_f=\Lambda\,U\,\Gamma^\prime\,U^T\,\Lambda,
\end{equation}
where $\Gamma^\prime$
is the $n\times n$ diagonal matrix defined by
$$
\gamma_i^\prime
=B_{1/f}^{\ast}(x_i)
=\sum_{d\Vert x_i\atop {d\,\not\>\Vert\,x_t\atop x_t<x_i}}
({1\over f}\oplus\mu^\ast)(d).
$$

The following factorizations may be considered well known \cite{BL93,BL95,KH}.

Let $S=\{x_1,x_2,\ldots, x_n\}$ be a set of distinct positive integers, and
let $S_{d}=\{w_1, w_2,\ldots, w_r\}$ be the set of all (positive) divisors of the elements of $S$ (that is, $S_{d}$ is the divisor closure of $S$).
Let $f$ be an  arithmetical function such that $(f\star\mu)(w_i)>0$ for all
$w_i\in S_{d}$.
Then
\begin{equation}\label{eq: N}
(S)_f=N N^T,
\end{equation}
where $N$ is the $n\times r$ matrix defined by
$$
n_{ij}=
\begin{cases}
\sqrt{(f\star\mu)(w_j)}\ \ {\rm if}\  \ w_j\,\vert\, x_i,\\
0\ \  {\rm otherwise}.
\end{cases}
$$
Let $f$ be a quasi-multiplicative function such that $f(w_i)\ne 0$ and
 $((1/f)\star\mu)(w_i)>0$ for all $w_i\in S_{d}$.
Then
\begin{equation}\label{eq: Np}
[S]_f=N' (N')^T,
\end{equation}
where $N'$ is the $n\times r$ matrix defined by
$$
n'_{ij}=
\begin{cases}
f(x_i)  \sqrt{((1/f)\star\mu)(w_j)}\ \ {\rm if}\  \ w_j\,\vert\, x_i,\\
0\ \  {\rm otherwise}.
\end{cases}
$$

We next present some new factorizations; although the ideas are well known \cite{KH}.

Let $S=\{x_1,x_2,\ldots, x_n\}$ be a set of distinct positive integers, and
let $S_{ud}=\{w_1, w_2,\ldots, w_r\}$ be the set of all unitary divisors of the elements of $S$ (that is, $S_{ud}$ is the unitary divisor closure of $S$).
Let $f$ be an  arithmetical function such that $(f\oplus\mu^\ast)(w_i)>0$ for all
$w_i\in S_{ud}$.
Then
\begin{equation}\label{eq: M}
(S^{\ast\ast})_f=M M^T,
\end{equation}
where $M$ is the $n\times r$ matrix defined by
$$
m_{ij}=
\begin{cases}
\sqrt{(f\oplus\mu^\ast)(w_j)}\ \ {\rm if}\  \ w_j\,\Vert\, x_i,\\
0\ \  {\rm otherwise}.
\end{cases}
$$
Let $f$ be a completely multiplicative function such that $f(w_i)\ne 0$ and
 $((1/f)\oplus\mu^\ast)(w_i)>0$ for all $w_i\in S_{ud}$.
Then
\begin{equation}\label{eq: Mp}
[S^{\ast\ast}]_f=M' (M')^T,
\end{equation}
where $M'$ is the $n\times r$ matrix defined by
$$
m'_{ij}=
\begin{cases}
f(x_i)  \sqrt{((1/f)\oplus\mu^\ast)(w_j)}\ \ {\rm if}\  \ w_j\,\Vert\, x_i,\\
0\ \  {\rm otherwise}.
\end{cases}
$$

\section{Inertia} \label{se: inertia}

The inertia of a Hermitean matrix $H$ is the triple $(i_{+}(H), i_{-}(H), i_0(H))$, where
$i_{+}(H)$,  $i_{-}(H)$ and $i_0(H)$ are the numbers of positive, negative and zero eigenvalues of the matrix $H$, counting multiplicities \cite{HJ}.
The factorizations in Section \ref{se: facto} make it possible to consider the inertias of the matrices $(S)_f$, $[S]_f$, $(S^{\ast\ast})_f$ and $[S^{\ast\ast}]_f$.
Inertia of MIN and MAX matrices is considered in \cite{MH2}.

\begin{theorem}\label{th: inertiamS}
Let $f(x)=x^m$, where $m>0$.
\begin{itemize}

\item[\rm(a)] If $S$ is any set of $n$ distinct positive integers, then

$i_{+}((S)_f)=n$, $i_{-}((S)_f)=i_{0}((S)_f)=0$.

\item[\rm(b)] If  $S$ is factor-closed, then

$i_{+}([S]_f)$ is the number of $x_i\in S$ such that $\omega(x_i)$ is even,

$i_{-}([S]_f)$ is the number of $x_i\in S$ such that $\omega(x_i)$ is odd,

$i_{0}([S]_f)=0$.

Here $\omega(x_i)$ is the number of distinct prime factors of $x_i$ with $\omega(1)=0$.

\end{itemize}

\end{theorem}

\begin{proof}
(a)  In this case
$$
(f\star\mu)(w_i)=J_m(w_i)=\sum_{d\Vert w_i} d^m \mu(w_i/d).
$$
The values of Jordan's totient $J_m$ at prime powers are given as
$$
J_m(p^a)=p^{am}-p^{(a-1)m}>0,
$$
and since $J_m$ is multiplicative, all the values of $J_m$ are positive.
Thus we may  apply Equation \eqref{eq: N},
which shows that $(S)_f$ is positive definite.
Thus all eigenvalues are positive.

(b) Let $S$ be factor-closed. Factorization \eqref{eq: deltap} implies that $[S]_f$ is ${}^T$congruent with the diagonal matrix $\Delta^\prime$; hence it suffices to consider the diagonal matrix $\Delta^\prime$, whose diagonal elements in the case of factor-closed set  are
$$
A_{1/f}(x_i)
=((1/f)\star\mu)(x_i)
=J_{-m}(x_i)
=\sum_{d\vert x_i} d^{-m} \mu(x_i/d).
$$
The values of $J_{-m}$ at prime powers are given as
$$
J_{-m}(p^a)=p^{-am}-p^{-(a-1)m}<0.
$$
Since  in this case $J_{-m}$ is multiplicative, we conclude that
$J_{-m}(x_i)$ is positive if and only if $\omega(x_i)$ is even, and
$J_{-m}(x_i)$ is negative if and only if $\omega(x_i)$ is odd.
This completes the proof.
\end{proof}

\begin{remark}
Theorem \ref{th: inertiamS} holds for multiplicative functions $f$ with $f(p^a)>f(p^{a-1})$ for all prime powers $p^a>1$.
\end{remark}

\begin{theorem}\label{th: inertia-mS}
Let $f(x)=x^{-m}$, where $m>0$.
\begin{itemize}

\item[\rm(a)]  If $S$ is any set of $n$ distinct positive integers, then

$i_{+}([S]_f)=n$, $i_{-}([S]_f)=i_{0}([S]_f)=0$.

\item[\rm(b)]  If $S$ is factor-closed, then

$i_{+}((S)_f)$ is the number of $x_i\in S$ such that $\omega(x_i)$ is even,

$i_{-}((S)_f)$ is the number of $x_i\in S$ such that $\omega(x_i)$ is odd,

$i_{0}((S)_f)=0$.

\end{itemize}

\end{theorem}

Proof is similar to that of Theorem \ref{th: inertiamS}, and we omit the details.

\begin{remark}
Theorem \ref{th: inertia-mS} holds for multiplicative functions $f$ with $0\ne f(p^a)<f(p^{a-1})$ for all prime powers $p^a>1$.
 \end{remark}

\begin{theorem}\label{th: inertiam}
Let  $f(x)=x^m$, where $m>0$.
\begin{itemize}

\item[\rm(a)]  If $S$ is  any set of $n$ distinct positive integers, then

$i_{+}((S^{\ast\ast})_f)=n$, $i_{-}((S^{\ast\ast})_f)=i_{0}((S^{\ast\ast})_f)=0$.

\item[\rm(b)]  If $S$ is  UD-closed, then

$i_{+}([S^{\ast\ast}]_f)$ is the number of $x_i\in S$ such that $\omega(x_i)$ is even,

$i_{-}([S^{\ast\ast}]_f)$ is the number of $x_i\in S$ such that $\omega(x_i)$ is odd,

$i_{0}([S^{\ast\ast}]_f)=0$.
\end{itemize}

\end{theorem}

\begin{proof}
 Proof is similar to that of Theorem \ref{th: inertiamS}. We, however, present the details.

(a) In this case
$$
(f\oplus\mu^\ast)(w_i)=J_m^{\ast}(w_i)=\sum_{d\Vert w_i} d^m \mu^\ast(w_i/d).
$$
The values of $J_m^{\ast}$ at prime powers are given as
$$
J_m^{\ast}(p^a)=p^{am}-1>0,
$$
and since $J_m^{\ast}$ is multiplicative, all the values of $J_m^{\ast}$ are positive.
Thus we may  apply Equation \eqref{eq: M}, which shows that $(S^{\ast\ast})_f$ is positive definite.
Thus all eigenvalues are positive.

(b) Factorization \eqref{eq: gammap} implies that $[S^{\ast\ast}]_f$ is
${}^T$congruent with the diagonal matrix $\Gamma^\prime$.
By Sylvester's law, see \cite[p. 223]{HJ}, ${}^T$congruence preserves inertia; hence it suffices to consider the diagonal matrix $\Gamma^\prime$, whose diagonal elements in the case of UD-closed set  are
$$
B_{1/f}^{\ast}(x_i)=((1/f)\oplus\mu^\ast)(x_i)
=J_{-m}^{\ast}(x_i)
=\sum_{d\Vert x_i} d^{-m} \mu^\ast(x_i/d).
$$
The values of $J_{-m}^{\ast}$ at prime powers in this case are given as
$$
J_{-m}^{\ast}(p^a)=p^{-am}-1<0.
$$
Since  in this case $J_{-m}^{\ast}$ is multiplicative,  we see that
$J_{-m}^{\ast}(x_i)$ is positive if and only if $\omega(x_i)$ is even, and
$J_{-m}^{\ast}(x_i)$ is negative if and only if $\omega(x_i)$ is odd.
This completes the proof.
\end{proof}

\begin{remark}
Theorem \ref{th: inertiam} holds for all multiplicative functions $f$ with
$f(p^a)>1$ for all prime powers $p^a>1$.
\end{remark}

\begin{theorem}\label{th: inertia-m}
Let $f(x)=x^{-m}$, where $m>0$.
\begin{itemize}

\item[\rm(a)]  If $S$ is  any set of $n$ distinct positive integers, then

$i_{+}([S^{\ast\ast}]_f)=n$, $i_{-}([S^{\ast\ast}]_f)=i_{0}([S^{\ast\ast}]_f)=0$.

\item[\rm(b)]  If $S$ is UD-closed, then

$i_{+}((S^{\ast\ast})_f)$ is the number of $x_i\in S$ such that $\omega(x_i)$ is even,

$i_{-}((S^{\ast\ast})_f)$ is the number of $x_i\in S$ such that $\omega(x_i)$ is odd,

$i_{0}((S^{\ast\ast})_f)=0$.
\end{itemize}

\end{theorem}

Proof is similar to that of Theorem \ref{th: inertiam}, and we omit the details.

\begin{remark}
Theorem \ref{th: inertia-m} holds for multiplicative functions $f$ with
$0\ne f(p^a)<1$ for all prime powers $p^a>1$.
\end{remark}

\section{Positive definite matrices} \label{se: pd}

Factorizations in Section  \ref{se: facto} and Sylvester's law of inertia make it possible to easily consider positive definiteness of GCD type matrices.

\begin{theorem}\label{th: pd-gcd}
If $S$ is GCD-closed and $f$ is any arithmetical function,
then $(S)_f$ is positive definite if and only if
$A_{f}(x_i)>0$ for all $i=1, 2,\ldots, n$.
\end{theorem}

\begin{proof}
Factorization \eqref{eq: delta} and Sylvester's law show  that $(S)_f$ is positive definite if and only if  the diagonal matrix $\Delta$ is positive definite, which holds exactly when the diagonal elements $A_{f}(x_i)$ are positive.
\end{proof}

\begin{theorem}\label{th: pd-lcm}
If $S$ is GCD-closed and $f$ is a quasi-multiplicative function with $f(x)\ne 0$ for all $x$,
then $[S]_f$ is positive definite if and only if
$A_{1/f}(x_i)>0$ for all $i=1, 2,\ldots, n$.
\end{theorem}

Proof is similar to that of Theorem \ref{th: pd-gcd} and utilizes factorization \eqref{eq: deltap}.
We omit the details.

\begin{remark}
Theorems \ref{th: pd-gcd} and \ref{th: pd-lcm} are known for meet
and join matrices \cite{MH}.
\end{remark}

\begin{theorem}\label{th: pdmS}
Let $S$ be any set of $n$ distinct positive integers,
and let $f(x)=x^m$, where $m>0$. Then
\begin{itemize}

\item[\rm(a)] $(S)_f$ is positive definite,

\item[\rm(b)]  $[S]_f$ is indefinite for $n\ge 2$.
\end{itemize}

\end{theorem}

\begin{proof}
(a) This is shown in Theorem \ref{th: inertiamS} (a).

(b) The first leading principal minor of $[S]_f$ is $x_1^m>0$, and
the second leading principal minor is
$x_1^m x_2^m-([x_1, x_2])^{2m}<0$.
This shows that $[S]_f$ is indefinite for $n\ge 2$.
\end{proof}

\begin{remark}
Theorem \ref{th: pdmS}(a)  holds for all arithmetical functions $f$ with
 $(f\star\mu)(w_i)>0$ for all $w_i\in S_{d}$.
Theorem \ref{th: pdmS}(b) holds for all strictly increasing arithmetical functions $f$.
\end{remark}

\medskip

Bhatia \cite{Bh} says that a positive semidefinite matrix $H$ with $h_{ij}\ge 0$ for all $i, j$ is
\emph{infinitely divisible} if the $m$th Hadamard (or entrywise) power of $H$ is
positive semidefinite for all $m\ge 0$.

\begin{corollary} \label{co: (S)}
The matrix $(S)$ is infinitely divisible.
\end{corollary}

\begin{remark}
Theorem \ref{th: pdmS}(a) is a known result \cite{BL93,BL93JNT}.
Corollary \ref{co: (S)} is also known \cite{Bh}.
Theorem \ref{th: pdmS}(b) is known for $m=1$ \cite{O}.
\end{remark}

\begin{theorem}\label{th: pd-mS}
Let $S$ be any set of $n$ distinct positive integers,
and let  $f(x)=x^{-m}$, where $m>0$. Then
\begin{itemize}

\item[\rm(a)] $[S]_f$ is positive definite,

\item[\rm(b)]  $(S)_f$ is indefinite for $n\ge 2$.

\end{itemize}
\end{theorem}

Proof is similar to that of Theorem \ref{th: pdmS} and utilizes
Theorem \ref{th: inertia-mS}. We omit the details.

\begin{remark}
Theorem \ref{th: pd-mS}(a) holds for all quasi-multiplicative functions with
$f(w_i)\ne 0$ and $((1/f)\star\mu)(w_i)>0$ for all $w_i\in S_{d}$.
Theorem \ref{th: pd-mS}(b) holds for all strictly decreasing arithmetical functions $f$.
\end{remark}

\begin{remark}
Theorem \ref{th: pd-mS}(a) is a known result \cite{BL95,HL}.
\end{remark}

\begin{corollary}
The Hadamard inverse of $[S]$ is infinitely divisible.
\end{corollary}

\begin{theorem}\label{th: pd-gcud}
If $S$ is GCUD-closed and $f$ is any arithmetical function,
then $(S^{\ast\ast})_f$ is positive definite if and only if
$B^{*}_{f}(x_i)>0$ for all $i=1, 2,\ldots, n$.
\end{theorem}

\begin{proof}
Factorization \eqref{eq: gamma} and Sylvester's law show  that $(S^{\ast\ast})_f$ is positive definite if and only if  the diagonal matrix $\Gamma$ is positive definite, which holds exactly when the diagonal elements $B^{*}_{f}(x_i)$ are positive.
\end{proof}

\begin{remark}
Theorem \ref{th: pd-gcud} is known for meet matrices \cite{MH}.
\end{remark}

\begin{theorem}\label{th: pd-lcum}
If $S$ is GCUD-closed and $f$ is a completely multiplicative function with $f(x)\ne 0$ for all $x$,
then $[S^{\ast\ast}]_f$ is positive definite if and only if
$B^{*}_{1/f}(x_i)>0$ for all $i=1, 2,\ldots, n$.
\end{theorem}

Proof is similar to that of Theorem \ref{th: pd-gcud} and utilizes factorization \eqref{eq: gammap}. We omit the details.

\begin{theorem}\label{th: pdm}
Let $S$ be any set of $n$ distinct positive integers,
and let $f(x)=x^m$, where $m>0$. Then
\begin{itemize}

\item[\rm(a)] $(S^{\ast\ast})_f$ is positive definite,

\item[\rm(b)] $[S^{\ast\ast}]_f$ is indefinite for $n\ge 2$.
\end{itemize}

\end{theorem}

\begin{proof}
 (a) This is shown in Theorem \ref{th: inertiam}(a).

(b) The first leading principal minor of $[S^{\ast\ast}]_f$ is $x_1^m>0$,
and the second leading principal minor is
$x_1^m x_2^m-([x_1, x_2]^{\ast\ast})^{2m}<0$.
This shows that $[S^{\ast\ast}]_f$ is indefinite for $n\ge 2$.
\end{proof}

\begin{corollary}
The matrix $(S^{\ast\ast})$ is infinitely divisible.
\end{corollary}

\begin{remark}
Theorem \ref{th: pdm}(a)  holds for all arithmetical functions $f$ with
 $(f\oplus\mu^\ast)(w_i)>0$ for all $w_i\in S_{ud}$.
Theorem \ref{th: pdm}(b) holds for all strictly increasing arithmetical functions $f$.
\end{remark}

\begin{theorem}\label{th: pd-m}
Let $S$ be any set of $n$ distinct positive integers,
and let $f(x)=x^{-m}$, where $m>0$. Then
\begin{itemize}

\item[\rm(a)] $[S^{\ast\ast}]_f$ is positive definite,

\item[\rm(b)]  $(S^{\ast\ast})_f$ is indefinite for $n\ge 2$.

\end{itemize}

\end{theorem}

Proof is similar to that of Theorem \ref{th: pdm}. We omit the details.

\begin{corollary}
The Hadamard inverse of $[S^{\ast\ast}]$ is infinitely divisible.
\end{corollary}

\begin{remark}
Theorem \ref{th: pd-m}(a) holds for all completely multiplicative function with
$f(w_i)\ne 0$ and $((1/f)\oplus\mu^\ast)(w_i)>0$ for all $w_i\in S_{ud}$.
Theorem \ref{th: pd-m}(b) holds for all strictly decreasing arithmetical functions $f$.
\end{remark}

\section{$\ell_p$ norms}\label{se: norms}

Norms of GCD and LCM matrices have not been studied much in the literature.
Some results are obtained in \cite{Al,ATH,Bo,H1,H2,H3,So1,So2,Ta,Tu}.

In this section we provide asymptotic formulas for the $\ell_p$ norms of the GCD matrix $((i, j)_{n\times n})$,
the LCM matrix $([i, j]_{n\times n})$, the GCUD matrix $((i, j)_{n\times n}^{\ast\ast})$, the pseudo-LCUM matrix
$([i, j]_{n\times n}^{\ast\ast})$ and the matrix $((i, j)_{n\times n}^{\ast})$.
Here $(i, j)^{\ast}$ stands for the semi-unitary greatest common divisor (SUGCD),
being the greatest divisor of $i$ which is a unitary divisor of $j$. See, e.g., \cite{HaS}.
We utilize known asymptotic formulas for arithmetical functions.

Let $p\in\Z^+$. The $\ell_p$ norm of an $n\times n$ matrix $M$ is defined as
$$
\Vert M\Vert_p
=\left(\sum_{i=1}^n \sum_{j=1}^n |m_{ij}|^p \right)^{1/p}.
$$

\subsection{Norms of GCD matrices}

It is known that
\begin{equation*}
\sum_{i,j\le x} (i,j) = Ax^2\log x+Bx^2+O(x^{1+\theta+\epsilon})
\end{equation*}
for every $\epsilon>0$, where $A:=1/\zeta(2)$,
\begin{equation*}
B:=\frac{1}{\zeta(2)}\left(2\gamma-\frac{1}{2}
-\frac{\zeta(2)}{2}-\frac{\zeta'(2)}{\zeta(2)}\right)
\end{equation*}
and $\theta$ is the exponent in Dirichlet's divisor problem. (Here $\gamma$ is Euler's constant and $\zeta$ is the Riemann $\zeta$-function.)
This asymptotic formula is equivalent to that deduced in \cite{ChiSit1985} for the sum $\sum_{i\le j\le x} (i,j)$.
See also \cite{HilTot2016,Tot2010}.

This means that for $p=1$,
\begin{equation}
\Vert ((i, j))\Vert_1 = An^2\log n+Bn^2+O(n^{1+\theta+\epsilon}).
\end{equation}

In \cite{H1} it is shown a more rough result, namely
$$
\Vert ((i, j))\Vert_1=O(n^2\log n).
$$

\begin{theorem} \label{Th_gcd_2} Let $p\ge 2$ be a fixed integer. Then
\begin{equation} \label{asymp_gcd_p}
\sum_{i, j\le x} (i,j)^p = C_p x^{p+1}+O(E_p(x)),
\end{equation}
where
\begin{equation*}
C_p:=\frac{2\zeta(p)-\zeta(p+1)}{(p+1)\zeta(p+1)}
\end{equation*}
and $E_p(x)=x^p$ for $p>2$ and $E_2(x)=x^2\log x$.
\end{theorem}

This formula can be obtained from general results of Cohen \cite{Coh1960i,Coh1962ii} established for sums $\sum_{a,b\le x} f((a,b))$, where
$f$ is a certain arithmetic function. However, we offer here an alternative approach to proof, which will be used for the next theorems, as well.

\begin{proof} Consider the Jordan function $J_p$.  We have
\begin{equation*}
S_p(x):= \sum_{i, j\le x} (i,j)^p = \sum_{i,j\le x} \sum_{d\mid (i,j)} J_p(d)= \sum_{\substack{da\le x\\ db\le x}} J_p(d).
\end{equation*}

Writing this into
\begin{equation*}
S_p(x)=  \sum_{d\le x } J_p(d) \left(\sum_{a\le x/d } 1 \right)^2,
\end{equation*}
and by applying usual estimates, we only obtain that $S_p(x)=O(x^{p+1})$, the main term being absorbed by the error term.
The idea is to change  the order of summation:
\begin{equation*}
S_p(x)= \sum_{a,b\le x} \sum_{d\le x/M} J_p(d),
\end{equation*}
where $M:=\max(a,b)$. By using the well known \cite[Th. 6.4]{Mc} formula
\begin{equation} \label{asympt_J_p}
\sum_{n\le x} J_p(n) = \frac1{(p+1)\zeta(p+1)}x^{p+1} +O(x^p),
\end{equation}
valid for any fixed $p\ge 2$, we obtain
\begin{equation*}
S_p(x)= \sum_{a,b\le x} \left(\frac1{(p+1)\zeta(p+1)} (x/M)^{p+1} +O((x/M)^p)\right)
\end{equation*}
\begin{equation} \label{1}
= \frac{x^{p+1}}{(p+1)\zeta(p+1)} \sum_{a,b\le x} \frac1{M^{p+1}} + O\left(x^p \sum_{a,b\le x} \frac1{M^p} \right).
\end{equation}

Here the first sum is
\begin{equation*}
\sum_{a,b\le x} \frac1{M^{p+1}} = 2 \sum_{a\le b\le x} \frac1{b^{p+1}}- \sum_{a=b\le x} \frac1{b^{p+1}}
\end{equation*}
\begin{equation*}
= 2 \sum_{b\le x} \frac1{b^{p+1}}\sum_{a\le b} 1 - \sum_{b\le x} \frac1{b^{p+1}}=
2 \sum_{b\le x} \frac1{b^p}- \sum_{b\le x} \frac1{b^{p+1}}
\end{equation*}
\begin{equation*}
= 2\left( \zeta(p) +O(\frac1{x^{p-1}})\right) - \left(\zeta(p+1) + O(\frac1{x^p})\right) = 2 \zeta(p)- \zeta(p+1) +O(\frac1{x^{p-1}}).
\end{equation*}

Similarly,
\begin{equation*}
\sum_{a,b\le x} \frac1{M^p} = 2 \sum_{a\le b\le x} \frac1{b^p}- \sum_{a=b\le x} \frac1{b^p} \ll \sum_{b\le x} \frac1{b^{p-1}},
\end{equation*}
which is $\ll \log x$ for $p=2$ and is $\ll 1$ for $p>2$,
see \cite[p. 70]{Ap}.
Inserting into \eqref{1} completes the proof (valid for any real $p\ge 2$).
\end{proof}

By applying Newton's generalized binomial theorem we obtain
$$
(\sum_{i, j\le x} (i,j)^p)^{1/p}=(C_px^{p+1})^{1/p}+O((x^{p+1})^{(1/p)-1}E_p(x))
$$
$$
=C_p^{1/p} x^{1+(1/p)}+O((x^{(1/p)-p}E_p(x)).
$$

Thus the $\ell_p$ norm of the $n\times n$ GCD matrix $((i, j))$ possesses the asymptotic formula
given in the next Corollary.

\begin{corollary} Let $p\ge 2$ be an integer. Then
\begin{equation}
\Vert ((i, j))\Vert_p=C_p^{1/p} n^{1+(1/p)}+O((n^{(1/p)-p}E_p(n)).
\end{equation}
\end{corollary}

In \cite{H1} it is shown a more rough result
$$
\Vert ((i, j))\Vert_p=O(n^{1+(1/p)})
$$
for $p\ge 2$.

\subsection{Norms of LCM matrices}

It is known that for every integer $p\ge 1$ one has
\begin{equation} \label{asymp_lcm_p}
\sum_{i, j\le x} [i, j]^p = D_p x^{2(p+1)}+O(x^{2p+1}(\log x)^{2/3}(\log\log x)^{4/3}),
\end{equation}
where
$$
D_p:=\frac{\zeta(p+2)}{(p+1)^2\zeta(p)},
$$
deduced in \cite [Th.\ 2]{IkeMat2014}. Applying Newton's generalized binomial theorem we obtain
$$
(\sum_{i, j\le n} [i, j]^p)^{1/p}
=(D_p n^{2(p+1)})^{1/p}+
O((n^{2(p+1)})^{(1/p)-1}n^{2p+1}(\log n)^{2/3}(\log\log n)^{4/3})
$$
$$
=D_p^{1/p} n^{2+(2/p)}+O((n^{(2/p)+1}(\log n)^{2/3}(\log\log n)^{4/3}).
$$

Thus the $\ell_p$ norm of the $n\times n$ LCM matrix $([i, j])$ possesses the asymptotic formula
\begin{equation}
\Vert ([i, j])\Vert_p=D_p^{1/p} n^{2+(2/p)}+O((n^{(2/p)+1}(\log n)^{2/3}(\log\log n)^{4/3})
\end{equation}
for $p\ge 1$.  In \cite{H1} it is shown a more rough result
$$
\Vert ([i, j])\Vert_p=O(n^{2+(2/p)})
$$
for $p\ge 1$.

In a similar way, having an asymptotic formula of type \eqref{asymp_gcd_p} or \eqref{asymp_lcm_p}, the
$\ell_p$ norm of the corresponding matrix can be easily estimated.

\subsection{Norms of SUGCD matrices}

Now consider the SUGCD matrix $((i, j)_{n\times n}^{\ast})$.

\begin{theorem}
\begin{equation*}
\sum_{i,j\le x} (i,j)^* = G x^2 \log x +O(x^2),
\end{equation*}
where
\begin{equation*}
G:= \zeta(2)^{-1} \prod_{q\in \PP} \left (1-\frac1{(q+1)^2} \right),
\end{equation*}
the product being over the primes $q$.
\end{theorem}

\begin{proof} We use that $d\mid\mid (i,j)^*$ if and only if $d\mid i$ and $d\mid\mid j$. By the property of the unitary Euler function
$J^*_1=\phi^*$,
\begin{equation*}
S^*(x):= \sum_{i,j\le x} (i,j)^* = \sum_{i,j\le x} \sum_{d\mid \mid (i,j)^*} \phi^*(d)
\end{equation*}
\begin{equation*}
= \sum_{i,j\le x} \sum_{\substack{d\mid i\\  d\mid \mid j}} \phi^*(d)
= \sum_{\substack{i=da\le x \\j=db\le x\\ (d,b)=1}} \phi^*(d)
\end{equation*}
\begin{equation*}
= \sum_{d\le x} \phi^*(d) \sum_{a\le x/d} 1 \sum_{\substack{b\le x/d\\ (b,d)=1}} 1.
\end{equation*}
Let $\sigma_s(n)= \sum_{d\mid n} d^s$. According to \cite[Lemma 2.1]{Tot1989}, for any fixed $k$ and any $\varepsilon>0$,
\begin{equation*}
\sum_{\substack{m\le x\\ (m,k)=1}} 1 = x \frac{\phi(k)}{k}+ O(x^{\varepsilon}\sigma_{-\varepsilon}(k)).
\end{equation*}
We deduce that
\begin{equation*}
S^*(x)= \sum_{d\le x} \phi^*(d) \left(\frac{x}{d}+ O(1)\right) \left(\frac{x}{d}\cdot \frac{\phi(d)}{d}+ O((\frac{x}{d})^{\varepsilon} \sigma_{-\varepsilon}(d)) \right)
\end{equation*}
\begin{equation*}
= x^2 \sum_{d\le x} \frac{\phi(d) \phi^*(d)}{d^3} +  O(x\sum_{d\le x} 1 )  + O(x^{1+\varepsilon}
\sum_{d\le x} \frac{\sigma_{-\varepsilon}(d)}{d^{\varepsilon}})
\end{equation*}
\begin{equation*}
= x^2 ( G\log x+ O(1)) + O(x^2) + O(x^{1+\varepsilon} x^{1-\varepsilon})
\end{equation*}
\begin{equation*}
= G x^2 \log x + O(x^2),
\end{equation*}
by using \cite[Lemmas 2.2, 3.4]{Tot1989}.
\end{proof}

\begin{remark}
Let
\begin{equation*}
P^*(n) =\sum_{k=1}^n (k,n)^*
\end{equation*}
be the unitary gcd-sum function. It is known (\cite[Th.\ 3.2]{Tot1989}) that
\begin{equation*}
\sum_{n\le x} P^*(n) = \frac{G}{2} x^2 \log x +O(x^2).
\end{equation*}
It follows that
\begin{equation*}
\sum_{m,n\le x} (m,n)^* \sim  2 \sum_{n\le x} P^*(n) \sim G x^2 \log x, \quad  x\to \infty.
\end{equation*}
On the other hand,
\begin{equation*}
\sum_{m,n\le x} (m,n)^* = \sum_{m\le n\le x} (m,n)^* + \sum_{n\le m\le x} (m,n)^* - \sum_{n\le x} (n,n)^*
\end{equation*}
\begin{equation*}
= \sum_{n\le x} P^*(n) + \sum_{m\le x} P^*_1(m) - \sum_{n\le x} n,
\end{equation*}
where
\begin{equation*}
P_1^*(n) =\sum_{k=1}^n (n,k)^*
\end{equation*}
is the ``dual'' unitary gcd-sum function. Hence
\begin{equation*}
\sum_{n\le x} P_1^*(n) \sim \frac{G}{2} x^2 \log x, \quad  x\to \infty.
\end{equation*}
\end{remark}

\begin{theorem} \label{Th_sugcd_2} Let $p\ge 2$ be a fixed integer. Then
\begin{equation*}
\sum_{i, j\le x} ((i,j)^*)^p = C^*_p x^{p+1}+O(E^*_p(x)),
\end{equation*}
where
\begin{equation*}
C^*_p:=\frac{\zeta(p+1)}{p+1} D^*_p \sum_{n=1}^{\infty} \frac{(n-1)g_p(n)+G_p(n)}{n^{p+1}},
\end{equation*}
\begin{equation} \label{D_t_star}
D^*_p:= \prod_{q\in \PP} \left(1-\frac2{q^{p+1}}+\frac1{q^{p+2}} \right),
\end{equation}
\begin{equation} \label{g_t}
g_p(n)= \prod_{q\mid n} \left(1-\frac1{q}\right) \left(1-\frac1{q^{p+1}}\right) \left(1-\frac2{q^{p+1}}+\frac1{q^{p+2}} \right)^{-1},
\end{equation}
the product being over the prime divisors $q$ of $n$,
\begin{equation*}
G_p(n)= \sum_{k=1}^n g_p(k),
\end{equation*}
and $E^*_p(x)=x^p$ for $p>2$ and $E^*_2(x)=x^2(\log x)^2$.
\end{theorem}

\begin{proof} We use the method of the proof of Theorem \ref{Th_gcd_2}, namely summation in reverse order.
Consider the unitary Jordan function $J^*_p$, already defined in Section \ref{se: inertia}. We have
\begin{equation*}
S^*_p(x):= \sum_{i, j\le x} ((i,j)^*)^p = \sum_{i,j\le x} \sum_{d\mid\mid (i,j)^*} J^*_p(d)
\end{equation*}
\begin{equation*}
= \sum_{i,j\le x}  \sum_{\substack{d\mid i\\ d\mid \mid j}} J^*_p(d)=
\sum_{\substack{da\le x\\ db\le x\\ (d,b)=1}} J^*_p(d)
\end{equation*}
\begin{equation*}
= \sum_{a,b\le x} \sum_{\substack{d\le x/M\\ (d,b)=1}} J^*_p(d),
\end{equation*}
where $M:=\max(a,b)$. We use the formula
\begin{equation} \label{form_unit_Jordan}
\sum_{\substack{n\le x\\(n,k)=1}} J^*_p(n) = \frac{\zeta(p+1)}{p+1} D^*_p x^{p+1} g_p(k)  + O(x^p\tau(k)),
\end{equation}
valid for any fixed $p\ge 2$, $k\in \Z^+$, where $\tau(k)=\sum_{d\mid k} 1$. The proof of \eqref{form_unit_Jordan} is similar to 
the proof of \eqref{asympt_J_p}. See also \cite{Coh1961}. Note that $0<g_p(k)<1$ holds for every $k\in \Z^+$.

We obtain
\begin{equation*}
S^*_p(x)= \sum_{a,b\le x} \left(\frac{\zeta(p+1)}{p+1}D^*_p g_p(b) (x/M)^{p+1} +O((x/M)^p\tau(b))\right)
\end{equation*}
\begin{equation} \label{2}
= x^{p+1} \frac{\zeta(p+1)}{p+1}D^*_p \sum_{a,b\le x} \frac{g_p(b)}{M^{p+1}} + O\left(x^p \sum_{a,b\le x} \frac{\tau(b)}{M^p} \right).
\end{equation}

Here
\begin{equation*}
\sum_{a,b\le x} \frac{g_p(b)}{M^{p+1}} =  \sum_{a\le b\le x} \frac{g_p(b)}{b^{p+1}} + \sum_{b\le a\le x} \frac{g_p(b)}{a^{p+1}}-
\sum_{a=b\le x} \frac{g_p(b)}{b^{p+1}} =:S_1+S_2-S_3,
\end{equation*}
say. We deduce
\begin{equation*}
S_1 = \sum_{b\le x} \frac{g_p(b)} {b^{p+1}} \sum_{a\le b} 1 = \sum_{b\le x} \frac{g_p(b)} {b^p} = \sum_{b=1}^{\infty} \frac{g_p(b)} {b^p}
+O(\frac1{x^{p-1}}),
\end{equation*}
the series being convergent since $0<g_p(k)<1$ for every $k\in \Z^+$.
\begin{equation*}
S_2=  \sum_{a\le x} \frac1{a^{p+1}} \sum_{b\le a}  g_p(b)=  \sum_{a\le x} \frac{G_p(a)}{a^{p+1}} = \sum_{a=1}^{\infty} \frac{G_p(a)}{a^{p+1}}
+O(\frac1{x^{p-1}}),
\end{equation*}
using that $0<G_p(k)<k$ for every $k\in \Z^+$.

Also,
\begin{equation*}
S_3=  \sum_{b=1}^{\infty} \frac{g_p(b)}{b^{p+1}} + O(\frac1{x^p}).
\end{equation*}

For the error term in \eqref{2},
\begin{equation*}
\sum_{a,b\le x} \frac{\tau(b)}{M^p} < \sum_{a\le b\le x} \frac{\tau(b)}{b^p} + \sum_{b\le a\le x} \frac{\tau(b)}{a^p} =
\sum_{b\le x} \frac{\tau(b)}{b^{p-1}} + \sum_{a\le x} \frac1{a^p} \sum_{b\le a} \tau(b)
\end{equation*}
\begin{equation*}
\ll \sum_{b\le x} \frac{\tau(b)}{b^{p-1}} + \sum_{a\le x} \frac{\log a}{a^{p-1}}
\end{equation*}
which is $\ll (\log x)^2$ for $p=2$ and is $\ll 1$ for $p>2$. Inserting into \eqref{2} completes the proof (valid for any real $p\ge 2$).
\end{proof}

\subsection{Norms of GCUD matrices}

Next consider the GCUD matrix $((i,j)_{n\times n}^{**})$.

\begin{theorem}
\begin{equation} \label{asympt_form}
\sum_{i,j\le x} (i,j)^{**} = F x^2 \log x +O(x^2),
\end{equation}
where
\begin{equation*}
F:= \zeta(2) \prod_{q\in \PP} \left (1-\frac{4}{q^2}+ \frac{4}{q^3} - \frac1{q^4} \right).
\end{equation*}
\end{theorem}

\begin{proof} Let
\begin{equation*}
P^{**}(n) =\sum_{k=1}^n (k,n)^{**}
\end{equation*}
be the bi-unitary gcd-sum function. It is known (\cite[Th.\ 3]{Tot2009}) that
\begin{equation} \label{asympt_P_star_star}
\sum_{n\le x} P^{**}(n) = \frac{F}{2} x^2 \log x +O(x^2).
\end{equation}

We have
\begin{equation*}
\sum_{i, j\le x} (i, j)^{**} = 2 \sum_{i\le j\le x} (i, j)^{**} - \sum_{j\le x} (j, j)^{**}
\end{equation*}
\begin{equation*}
= 2 \sum_{j\le x} P^{**}(j) - \sum_{j\le x} j,
\end{equation*}
and \eqref{asympt_form} is a direct consequence of \eqref{asympt_P_star_star}.
\end{proof}

\begin{theorem} Let $p\ge 2$ be a fixed integer. Then
\begin{equation*}
\sum_{i, j\le x} ((i,j)^{**})^p = C^{**}_p x^{p+1}+O(E^{**}_p(x)),
\end{equation*}
where
\begin{equation*}
C^{**}_p:=\frac{\zeta(p+1)}{p+1} D^*_p \sum_{n=1}^{\infty} \frac{2\overline{G}_p(n) - g_p(n^2)}{n^{p+1}},
\end{equation*}
$D^*_p$ and $g_p(n)$ are defined by \eqref{D_t_star} and \eqref{g_t}, respectively,
\begin{equation*}
\overline{G}_p(n)= \sum_{k=1}^n g_p(kn),
\end{equation*}
and $E^{**}_p(x)=x^p \log x$ for $p>2$ and $E^{**}_2(x)=x^2(\log x)^3$.
\end{theorem}

\begin{proof} Similar to the proofs of Theorems \ref{Th_gcd_2} and \ref{Th_sugcd_2}. We use that $d\mid\mid (i,j)^{**}$ if and only if
$d\mid\mid i$ and $d\mid\mid j$.
\begin{equation*}
S^{**}_p(x):= \sum_{i, j\le x} ((i,j)^{**})^p = \sum_{i,j\le x} \sum_{d\mid\mid (i,j)^{**}} J^*_p(d)
\end{equation*}
\begin{equation*}
= \sum_{i,j\le x}  \sum_{\substack{d\mid\mid i\\ d\mid \mid j}} J^*_p(d)=
\sum_{\substack{da\le x\\ db\le x\\ (d,ab)=1}} J^*_p(d)
\end{equation*}
\begin{equation*}
= \sum_{a,b\le x} \sum_{\substack{d\le x/M\\ (d,ab)=1}} J^*_p(d),
\end{equation*}
where $M:=\max(a,b)$. We use formula \eqref{form_unit_Jordan} and obtain that
\begin{equation*}
S^{**}_p(x)= \sum_{a,b\le x} \left(\frac{\zeta(p+1)}{p+1}D^*_p g_p(ab) (x/M)^{p+1} +O((x/M)^p\tau(ab))\right)
\end{equation*}
\begin{equation} \label{33}
= x^{p+1} \frac{\zeta(p+1)}{p+1}D^*_p \sum_{a,b\le x} \frac{g_p(ab)}{M^{p+1}} + O\left(x^p \sum_{a,b\le x} \frac{\tau(ab)}{M^p} \right).
\end{equation}

Here (having again symmetry in the variables $a$ and $b$),
\begin{equation*}
\sum_{a,b\le x} \frac{g_p(ab)}{M^{p+1}} =  2 \sum_{a\le b\le x} \frac{g_p(ab)}{b^{p+1}} -
\sum_{a=b\le x} \frac{g_p(b^2)}{b^{p+1}}
\end{equation*}
\begin{equation*}
= 2 \sum_{b\le x} \frac{\overline{G}_p(b)}{b^{p+1}}- \sum_{b\le x} \frac{g_p(b^2)}{b^{p+1}} = 2 \sum_{b=1}^{\infty}
\frac{\overline{G}_p(b)}{b^{p+1}} - \sum_{b=1}^{\infty} \frac{g_p(b^2)}{b^{p+1}} + O(\frac1{x^{p-1}}).
\end{equation*}

For the error term in \eqref{33}, use that $\tau(ab)\le \tau(a)\tau(b)$ for any $a,b\in \Z^+$.
\begin{equation*}
\sum_{a,b\le x} \frac{\tau(ab)}{M^p} \ll \sum_{a\le b\le x} \frac{\tau(a)\tau(b)}{b^p} =\sum_{b\le x} \frac{\tau(b)}{b^p}\sum_{a\le b} \tau(a)
\end{equation*}
\begin{equation*}
\ll \sum_{b\le x} \frac{\tau(b)}{b^p} b\log b \ll (\log x) \sum_{b\le x} \frac{\tau(b)}{b^{p-1}},
\end{equation*}
which is $\ll (\log x)^3$ for $p=2$ and is $\ll \log x$ for $p>2$, see \cite[p. 70]{Ap}. The proof works for any real $p\ge 2$.
\end{proof}

\subsection{Norms of pseudo-LCUM matrices}

Finally, consider the pseudo-LCUM matrix $([i,j]^{**}_{n\times n})$.

\begin{theorem} Let $p\ge 1$ be an integer. Then
\begin{equation*}
\sum_{i, j\le x} ([i,j]^{**})^p = \frac{\beta_p}{(p+1)^2} x^{2(p+1)}+O(x^{2p+1}(\log x)^2),
\end{equation*}
where
\begin{equation*}
\beta_p:=\zeta(2)\zeta(p+2) \prod_{q\in \PP} \left(1-\frac{2}{q^2} + \frac{2}{q^3} - \frac1{q^4} - \frac{2}{q^{p+3}}+ \frac{2}{q^{p+4}} \right).
\end{equation*}
\end{theorem}

\begin{proof} Let $\id_s(n)=n^s$.
\begin{equation*}
U_p(x):= \sum_{i, j\le x} ([i,j]^{**})^p = \sum_{i, j\le x} \left(\frac{ij}{(i,j)^{**}}\right)^p =
\sum_{i,j\le x} (ij)^p \sum_{d\mid\mid (i,j)^{**}} (\mu^* \oplus \id_{-p})(d).
\end{equation*}

Let denote $h_p(n)=(\mu^* \oplus \id_{-p})(n)$, which is multiplicative and $h_p(q^\nu)=1/q^{\nu p}-1$ for every prime power
$q^\nu$ ($\nu \ge 1$). Hence $|h_p(n)|\le 1$ for every $n\in \Z^+$ (and every real $p>0$). We have
\begin{equation*}
U_p(x)=  \sum_{\substack{da\le x\\ db\le x\\ (d,ab)=1}} (d^2ab)^p h_p(d) = \sum_{d\le x} d^{2p} h_p(d)
\left( \sum_{\substack{a\le x/d\\ (a,d)=1}} a^p\right)^2.
\end{equation*}

We use the known \cite[Lemma 2.1]{Tot1989} formula
\begin{equation*}
\sum_{\substack{n\le x\\ (n,k)=1}} n^p = \frac{x^{p+1}}{p+1} \frac{\phi(k)}{k}+ O(x^p\tau(k)),
\end{equation*}
valid for every real $p\ge 0$ and $k\in \Z^+$, and obtain
\begin{equation*}
U_p(x)= \sum_{d\le x} d^{2p} h_p(d) \left(\frac{(x/d)^{p+1}}{p+1}\cdot \frac{\phi(d)}{d} + O((x/d)^p\tau(d))\right)^2
\end{equation*}
\begin{equation*}
= \sum_{d\le x} d^{2p} h_p(d) \left(\frac{(x/d)^{2(p+1)})}{(p+1)^2}\cdot \frac{\phi^2(d)}{d^2} + O((x/d)^{2p+1}\tau(d))\right)
\end{equation*}
\begin{equation} \label{last}
= \frac{x^{2(p+1)}}{(p+1)^2} \sum_{d\le x} \frac{h_p(d)\phi^2(d)}{d^4} + O\left(x^{2p+1} \sum_{d\le x} \frac{|h_p(d)|\tau(d)}{d}
\right).
\end{equation}

The sum of the main term in \eqref{last} can be written as
\begin{equation*}
\sum_{d=1}^{\infty} \frac{h_p(d)\phi^2(d)}{d^4}+O(\frac1{x})=\beta_p+O(\frac1{x}),
\end{equation*}
the series being convergent since $h_p(n)$ is bounded, where $\beta_p$ can be easily computed by the Euler product formula.
The error term in \eqref{last} is $\ll x^{2p+1}(\log x)^2$, see \cite[p. 70]{Ap}.  Note that
the proof is valid for any positive real~$p$.
\end{proof}


\end{document}